\documentclass[12pt]{article}

\usepackage{amsmath}
\usepackage{amsthm}
\usepackage{amssymb}
\usepackage{graphicx}
\usepackage{standalone}
\usepackage{comment}
\usepackage{color}
\usepackage{enumerate,hyperref}

\usepackage{xcolor}

\usepackage[lite]{amsrefs}

\makeatletter
\usepackage{comment}
\let\wfs@comment@comment\comment
\let\comment\@undefined

\usepackage{changes}
\let\wfs@changes@comment\comment
\let\comment\@undefined

\newcommand\comment{%
    \ifthenelse{\equal{\@currenvir}{comment}}
    {\wfs@comment@comment}
    {\wfs@changes@comment}%
}

\usepackage{amssymb}

\usepackage[all,cmtip]{xy}

\usepackage{mathrsfs}

\usepackage{tabularx}
\usepackage{booktabs}
\usepackage[labelfont=bf,format=plain,justification=raggedright,singlelinecheck=false]{caption}

\theoremstyle{plain}
\newtheorem{thm}{Theorem}[section]

\newtheorem{lem}[thm]{Lemma}

\newtheorem{defi}[thm]{Definition}

\newtheorem{cor}[thm]{Corollary}

\author{Ferdinando Zullo}
\title{Saturating linear sets in PG$(2,q^4)$}
\date{\today}

\newcommand{\F}{\mathbb F}

\newcommand{\fq}{{\mathbb F}_q}

\newcommand{\la}{\langle}
\newcommand{\ra}{\rangle}

\DeclareMathOperator{\PG}{PG}

\DeclareMathOperator{\C}{\mathcal{C}}

\theoremstyle{definition}
\newtheorem{definition}[thm]{Definition}

\newtheorem{rem}[thm]{Remark}
\newtheorem{ex}[thm]{Example}

\begin{document}

\maketitle

\begin{abstract}
Bonini, Borello and Byrne started the study of saturating linear sets in Desarguesian projective spaces, in connection with the covering problem in the rank metric. In this paper we study \emph{$1$-saturating} linear sets in PG$(2,q^4)$, that is $\fq$-linear sets in PG$(2,q^4)$ with the property that their secant lines cover the entire plane. By making use of a characterization of generalized Gabidulin codes, we prove that the rank of such a linear set is at least $5$.
This answers to a recent question posed by Bartoli, Borello and Marino.
\end{abstract}

\section*{Introduction}

Consider an Hamming-metric linear code $\mathcal{C}$ in $\fq^n$. $\mathcal{C}$ is called \textbf{$\rho$-covering} if for any element $y$ in $\fq^n$ there exists at least one codeword $c \in \C$ such that $d(y,c)\leq \rho$. The \textbf{covering radius} of $\C$ is the smallest value $\rho$ for which $\C$ turns out to be a $\rho$-covering code. To determine this value is a very difficult problem in general and indeed a classical problem in coding theory is the \emph{covering problem}, that is the determination of the smallest the minimal size of a $\rho$-covering code of length $n$. This problem has been deeply investigated during the last years, see \cite{book,DMP03,Den22,Zoli} for some references.
One of the main tool used to study the covering problem has been the geometric interpretation of the covering radius. Indeed, equivalence classes of $(\rho+1)$-covering codes of dimension $k$ are in one-to-one correspondence with equivalence classes of $\rho$-saturating sets in PG$(k-1,q)$.
A \textbf{$\rho$-saturating set} is a pointset $S$ in PG$(k-1,q)$ such that every point of PG$(k-1,q)$ lies in at least one subspace generated by $\rho+1$ points of $S$. 
There is a notable paucity of literature exploring the rank-metric, with only a few papers appearing on this subject; see 
\cite{BBM23,BoBM23,ByRav17,GY08}. In \cite{BoBM23}, Bonini, Borello and Byrne provide a geometric description of covering codes in terms of saturating linear sets, by making use of the correspondence shown in \cite{Ra} and in \cite{ABNR22}.
An $\fq$-linear set $L_U \subseteq \mathrm{PG}(k-1,q^m)$ is said to be \textbf{$\rho$-saturating} if any point in $\mathrm{PG}(k-1,q^m)$ belongs to at least one subspaces of dimension $\rho$ spanned by $\rho+1$ points of $L_U$. 
The main problem is the determination of the smallest rank of a $\rho$-saturating $\fq$-linear set in PG$(k-1,q^m)$.
This value is denoted by $s_{q^m/q}(k,\rho)$. 
In \cite{BoBM23} the following lower bounds have been provided 
\begin{equation}\label{eq:boundsons}
s_{q^m/q}(k,\rho)\geq
    \begin{cases}
        \left\lfloor \frac{mk}{\rho} \right\rfloor -m+\rho & \text{if } q>2,\\
        \\
        \left\lfloor \frac{mk-1}{\rho} \right\rfloor -m+\rho & \text{if } q=2 \text{ and } \rho>1,\\
        \\
        m(k-1)+1 & \text{if } q=2 \text{ and } \rho=1.\\
    \end{cases}
\end{equation}
By making use of some subgeometries configurations, in \cite{BBM23,BoBM23} it has been proved that the first bound is tight when $\rho$ divides $k$. A careful study of scattered linear sets with respect to subspaces, in \cite{BBM23} Bartoli, Borello and Marino have provided more bounds; see \cite[Section 2]{BBM23} and \cite[Remark 4.1]{BBM23} for a complete list of the known values of $s_{q^m/q}(k,\rho)$. 
In \cite[Section 3]{BBM23} the authors point out that
\[ 4 \leq s_{q^4/q}(3,2)\leq 5,\]
and, with a meticulous analysis of certain algebraic varieties, they prove that when $q$ is large and even $s_{q^4/q}(3,2)= 5$. 
In \cite[Question 4.2]{BBM23}, they ask whether or not $s_{q^4/q}(3,2)= 5$ in the remaining cases.
In this paper we answer to this question with a different approach, proving that $s_{q^4/q}(3,2)= 5$ for any prime power $q$. Finally, we also characterize the linear sets of minimum rank giving saturating linear sets in PG$(2,q^4)$.

The paper is organized as follows. 
Section \ref{s:preliminaries} is devoted to recall some results on linear sets, rank-metric codes and linearized polynomials we will need to prove the main result.
In Section \ref{sec:gab} we recall the construction of generalized Gabidulin codes and some characterization results, for which we will give a more geometric interpretation. In the last section, Section \ref{sec:main}, we prove the main result of this paper.

\section{Preliminaries}\label{s:preliminaries}

Let $p$ be a prime and let $h$ be a positive integer. We fix $q=p^h$ and denote by $\fq$ the finite field with $q$ elements. Let $V$ be an $\fq$-vector space, $\langle U \rangle_{\F_{q}}$ denotes the $\fq$-span of $U$, with $U$ a subset of $V$. Moreover, $\PG(k-1,q)$ denotes the projective Desarguesian space of dimension $k-1$ and order $q$ and $\PG(V,\fq)$ denotes the projective space obtained by $V$. Finally, $\mathrm{GL}(k,q)$ denotes the general linear group.

\subsection{Linear sets}

Let $V$ be a $k$-dimensional vector space over $\F_{q^m}$ and let $\Lambda=\PG(V,\F_{q^m})=\PG(k-1,q^m)$.
Recall that, if $U$ is an $\fq$-subspace of $V$ of dimension $n$, then the set of points
\[ L_U=\{\la { u} \ra_{\mathbb{F}_{q^m}} : { u}\in U\setminus \{{ 0} \}\}\subseteq \Lambda \]
is said to be an $\fq$-\textbf{linear set of rank $n$}.

\begin{defi}
Let $\Omega=\PG(W,\F_{q^m})$ be a projective subspace of $\Lambda$. The \textbf{weight of $\Omega$} in $L_U$ is defined as 
\[ w_{L_U}(\Omega)=\dim_{\fq}(U\cap W). \]
\end{defi}

For any $i \in \{0,\ldots,n\}$, denote by $N_i$ the number of points of $\Lambda$ of weight $i$ with respect to $L_U$, then
they satisfy the following relations, when $L_U \ne \Lambda$:
\begin{equation}\label{eq:card}
    |L_U| \leq \frac{q^n-1}{q-1},
\end{equation}
\begin{equation}\label{eq:pesicard}
    |L_U| =N_1+\ldots+N_n,
\end{equation}
\begin{equation}\label{eq:pesivett}
    N_1+N_2(q+1)+\ldots+N_n(q^{n-1}+\ldots+q+1)=q^{n-1}+\ldots+q+1,
\end{equation}
\begin{equation}\label{eq:size}
    \text{if } L_U\ne \emptyset \text{ then } |L_U|\equiv 1 \pmod{q}.
\end{equation}
Moreover, the following holds
\begin{equation}\label{eq:wpointsrank}
    w_{L_U}(P)+w_{L_U}(Q)\leq n,
\end{equation}
for any $P,Q \in \mathrm{PG}(k-1,q^m)$ with $P\ne Q$.

Furthermore, $L_U$ and $U$ are called \textbf{scattered} if $L_U$ has the maximum number $\frac{q^n-1}{q-1}$ of points, or equivalently, if all points of $L_U$ have weight one. 
Scattered linear sets, and more generally scattered spaces, have been introduced in \cite{BL2000} and some generalizations have been investigated. Indeed, the following definition was given in \cite{CsMPZ}, extending those in \cite{Lunardon2017,JohnGeertrui}.
An $\F_q$-linear set $L_U$ of $V$ is called $h$-\textbf{scattered} (or scattered w.r.t.\ the $(h-1)$-dimensional projective subspaces) if $\la U \ra_{\F_{q^m}}=V$ and the weight of any projective subspace of dimension $h-1$ is at most $h$.

Clearly, for a non-empty linear set $L_U$, \eqref{eq:size} can be improved if some assumptions are added. 

\begin{thm}[{\cite[Theorem 1.2]{DeBeuleVdV}} and {\cite[Lemma 2.2]{BoPol}}]\label{thm:sizelinset}
 If $L_U$ is an $\fq$-linear set of rank $n$, with $1 < n \leq m$ on $\PG(1,q^m)$, and $L_U$ contains at least one point of weight $1$, then $|L_U| \geq q^{n-1} + 1$.
\end{thm}

We also refer to \cite{ASmin,CsMP23} for bounds on the size of a linear set.

The following is a result taken from \cite{AMSZ23} which will be useful to compute the size of a linear set of rank $n$ if it admits a line of weight $n-1$.

\begin{lem}\cite[Lemma 3.1]{AMSZ23} \label{lem:extensionpoint}
Let $U$ be a $k$-dimensional $\F_q$-subspace of $V$ and let $v \in V$ be such that $\langle v \rangle_{\F_{q^n}} \notin L_{U} $. Let $U_1=U \oplus \langle v \rangle_{\F_q}$. Then any point in $L_{U_1} \setminus L_{U}$ has weight one in $L_{U_1}$. Moreover, if $\langle v \rangle_{\F_{q^n}} \cap \langle U \rangle_{\F_{q^n}}=\{0\}$, then $\lvert L_{U_1} \rvert = \lvert L_{U} \rvert+q^k$.
\end{lem}

\subsection{Rank-metric codes}\label{sec:rankmetric}

Rank-metric codes were introduced by Delsarte \cite{Delsarte} in 1978 as subsets of matrices and recently they have been intensively investigated due to their applications; see e.g. \cite{app,sheekey_newest_preprint,PZ}.
In this section we will be interested in rank-metric codes in $\F_{q^m}^n$ where the \textbf{rank} (weight) $w(v)$ of a vector $v=(v_1,\ldots,v_n) \in \F_{q^m}^n$ is  
\[w(v)=\dim_{\fq} (\langle v_1,\ldots, v_n\rangle_{\fq}).\] 
The \textbf{rank distance} is defined as $d(x,y)=w(x-y)$, where $x, y \in \F_{q^m}^n$. 

A \textbf{linear rank-metric code} $\C $ is an $\F_{q^m}$-subspace of $\F_{q^m}^n$ endowed with the rank metric.
Let $\C \subseteq \F_{q^m}^n$ be a linear rank-metric code. We will write that $\C$ is an $[n,k,d]_{q^m/q}$ code (or $[n,k]_{q^m/q}$ code) if $k=\dim_{\F_{q^m}}(\C)$ and $d$ is its minimum distance, that is $d=\min\{d(x,y) \colon x, y \in \C, x \neq y  \}$.

From the classification of $\F_{q^m}$-linear isometry of $\F_{q^m}^n$ (see e.g. \cite{berger2003isometries}) the following definition arises. We say that two rank-metric codes $\C,\C' \subseteq \F_{q^m}^n$ are \text{equivalent} if and only if there exists a matrix $A \in \mathrm{GL}(n,q)$ such that
\[\C'=\C A=\{vA : v \in \C\}.\] 

As for the classical Hamming-metric codes, for rank-metric codes, Delsarte showed a Singleton-like bound.

\begin{thm}\cite{Delsarte} \label{th:singletonrank}
    Let $\C \subseteq \F_{q^m}^n$ be an $[n,k,d]_{q^m/q}$ code.
Then 
\begin{equation}\label{eq:boundgen}
mk \leq \max\{m,n\}(\min\{m,n\}-d+1).\end{equation}
\end{thm}

An $[n,k,d]_{q^m/q}$ code is called \textbf{Maximum Rank Distance code} (or shortly \textbf{MRD code}) if its parameters attains the bound \eqref{eq:boundgen}. In the next section we will see some example of MRD codes.

One of the main tool of this paper is given by the geometric description of linear rank-metric codes.

The geometric counterpart of non-degenerate rank-metric codes (that is, the columns of any generator matrix of $\C$ are $\fq$-linearly independent) are the systems. 
An $[n,k,d]_{q^m/q}$ \textbf{system} $U$ is an $\F_q$-subspace of $\F_{q^m}^k$ of dimension $n$, such that
$ \langle U \rangle_{\F_{q^m}}=\F_{q^m}^k$ and
$$ d=n-\max\left\{\dim_{\F_q}(U\cap H) \mid H \textnormal{ is an $\F_{q^m}$-hyperplane of }\F_{q^m}^k\right\}.$$

Moreover, two $[n,k,d]_{q^m/q}$ systems $U$ and $U'$ are \textbf{equivalent} if there exists an $\F_{q^m}$-isomorphism $\varphi\in\mathrm{GL}(k,q^m)$ such that
$$ \varphi(U) = U'.$$

\begin{thm}\cite{Ra} \label{th:connection}
Let $\C$ be a non-degenerate $[n,k,d]_{q^m/q}$ rank-metric code and let $G$ be an its generator matrix.
Let $U \subseteq \F_{q^m}^k$ be the $\F_q$-span of the columns of $G$.
The rank weight of an element $x G \in \C$, with $x \in \F_{q^m}^k$ is
\begin{equation}\label{eq:relweight}
w(x G) = n - \dim_{\fq}(U \cap x^{\perp}),\end{equation}
where $x^{\perp}=\{y \in \F_{q^m}^k \colon x \cdot y=0\}$, where $\cdot$ denotes the standard inner product. In particular,
\begin{equation} \label{eq:distancedesign}
d=n - \max\left\{ \dim_{\fq}(U \cap H)  \colon H\mbox{ is an } \F_{q^m}\mbox{-hyperplane of }\F_{q^m}^k  \right\}.
\end{equation}
\end{thm}

Actually, the above result allows us to give a one-to-one correspondence between equivalence classes of non-degenerate $[n,k,d]_{q^m/q}$ codes and equivalence classes of $[n,k,d]_{q^m/q}$ systems, see \cite{Ra}.
The system $U$ and the code $\C$ as in Theorem \ref{th:connection} are said to be \textbf{associated}.

This connection can be used to show, for instance, that when $n\leq m$ the systems associated to MRD codes are the subspaces defining scattered linear sets with respect to the hyperplanes; see \cite{JohnGeertrui,Lunardon2017,MNT,ZZ}.

For further details we refer to \cite{PSSZ202x,PZ,Ra,sheekey_newest_preprint}.

\subsection{Linearized polynomials}

Any polynomial of the form  $f =\sum_{i=0}^{k} a_i x^{q^i}\in\F_{q^m}[x]$ is called \textbf{$q$-polynomial} (or \textbf{linearized polynomial}); if $a_k\neq0$ then the \textbf{$q$-degree} of $f $ is $k$.
The set of linearized polynomials over $\F_{q^m}$ will be denoted as $L_{m,q}$.
Such set, equipped with the operations of sum and multiplication by elements of $\fq$ and the composition, results to be an $\fq$-algebra.
The quotient algebra $\mathcal{L}_{m,q}=L_{m,q}/(x^{q^m}-x)$ is isomorphic to the $\fq$-algebra of $\fq$-linear
endomorphisms of $\F_{q^m}$.
Therefore we can identify the elements of $\mathcal{L}_{m,q}$ with the $q$-polynomials having $q$-degree smaller than $m$.
For an overview see \cite{wl,lidl_finite_1997}.

In order to find new examples of optimal codes in the rank metric, Sheekey in \cite{sheekey2016new} introduced the notion of scattered polynomial.

\begin{definition}
A polynomial $f \in {\mathcal{L}}_{m,q}$ is said to be \textbf{scattered} if for any $a,b \in \mathbb{F}_{q^m}^*$, $f(a)/a=f(b)/b$ implies that $a$ and $b$ are $\fq$-proportional.
\end{definition}

An example of scattered polynomial is given by $f(x)=x^q$.

\section{Generalized Gabidulin codes}\label{sec:gab}

In this section we present the family of generalized Gabidulin codes and some characterization results. We will also analyze the geometric consequences of such results.
Since the works of Delsarte in \cite{Delsarte} and Gabidulin in \cite{Gabidulin}, we know that linear MRD codes exist for any set of parameters. Indeed, they showed the first family of MRD codes, known as \emph{Gabidulin codes}. Later this family was generalized in \cite{kshevetskiy_new_2005}, obtaining the faimly of MRD codes called \emph{generalized Gabidulin codes}.
To define these codes we introduce the notion of Moore matrix. Let $\sigma$ be any automorphism of $\F_{q^m}$ fixing $\fq$, then for any vector $(v_1,\ldots,v_n)\in \F_{q^m}^n$ we define the \textbf{Moore matrix} of order $k\times n$ (with respect to $\sigma$) as
\[ M_{k,\sigma}(v_1,\ldots,v_n)=\begin{pmatrix}
v_1 & v_2 & \ldots & v_n\\
\sigma(v_1) & \sigma(v_2) & \ldots & \sigma(v_n)\\
& & \vdots & \\
\sigma^{k-1}(v_1) & \sigma^{k-1}(v_2) & \ldots & \sigma^{k-1}(v_n)
\end{pmatrix}.\]
Let $(v_1,\ldots,v_n)\in \F_{q^m}^n$ be such that $v_1,\ldots,v_n$ are $\fq$-linearly independent and let $\sigma$ as before. The code $\C\subseteq \F_{q^m}^n$ of dimension $k$ having $M_{k,\sigma}(v_1,\ldots,v_n)$ as a generator matrix is known as \textbf{generalized Gabidulin code}. When $\sigma$ is the Frobenius automorphism then we get the original examples found by Delsarte and Gabidulin, called simply \textbf{Gabidulin codes}.
Equivalently, these codes can be described via the evaluation of the polynomials in the subspace $\langle x,\sigma(x),\ldots,\sigma^{k-1}(x)\rangle_{\F_{q^m}}$ over $\fq$-linearly independent elements.
In \cite{HorlMarsh17}, Horlemann-Trautmann and Marshall provided some characterization results for MRD codes and for the family of generalized Gabidulin codes, whose approach has been used also for other families of codes; see \cite{GiuZ,NPH2}. 
The notion of equivalence of rank-metric codes considered in \cite{HorlMarsh17} is more general, but the same results can be proved with the equivalence considered in this paper.

More precisely, we are interested in the following results.

\begin{thm}\label{thm:characMRDGab}\cite[Theorem 5.1]{HorlMarsh17}
    If $\mathcal{C}$ is an linear MRD code in $\F_{q^m}^n$ of dimension either one or $n-1$, then $\C$ is equivalent to a Gabidulin code.
\end{thm}

\begin{thm}\label{thm:characMRDGab2}\cite[Corollary 5.2]{HorlMarsh17}
    If $\mathcal{C}$ is an linear MRD code in $\F_{q^m}^n$, with $n \in \{1,2,3\}$, then $\C$ is equivalent to a Gabidulin code.
\end{thm}

\begin{rem}
    One could also derive Theorem \ref{thm:characMRDGab} by first looking at the one-dimensional Gabidulin codes and then using that the dual of a Gabidulin code is a Gabidulin code as well.
\end{rem}

Using the connection described in Section \ref{sec:rankmetric}, the above results read as follows.

\begin{cor}\label{cor:inslinGab}
Consider an $\fq$-linear set $L_U$ in $\mathrm{PG}(k-1,q^m)$ of rank $n$ which is scattered with respect to the hyperplanes. Then
\begin{itemize}
    \item if $k=n-1$, then $U$ is $\mathrm{GL}(n-1,q^m)$-equivalent to $\{ (x,x^q,\ldots,x^{q^{n-2}})\colon x \in Z \}$ where $Z$ is any $\fq$-subspace of $\F_{q^m}$ of dimension $n$;
    \item if $n \in \{1,2,3\}$ and $k\leq n$, then $U$ is $\mathrm{GL}(k,q^m)$-equivalent to\\ $\{ (x,x^q,\ldots,x^{q^{k-1}})\colon x \in Z \}$, where $Z$ is any $\fq$-subspace of $\F_{q^m}$ of dimension $n$.
\end{itemize}
\end{cor}

\begin{rem}
When $k=2$, the above result implies that the scattered linear sets in PG$(1,q^m)$ of rank $3$ are all equivalent (via the action of GL$(2,q^m)$ on the subspaces representing the linear sets);  see also \cite[Proposition 5]{CsZa16} and \cite[Theorem 5]{LavVdV10}.
\end{rem}

\section{Saturating linear sets in PG$(2,q^4)$ of minimal rank}\label{sec:main}

A $1$-saturating linear set $L_U$ in PG$(2,q^4)$ is an $\fq$-linear set for which the secant lines to $L_U$ cover all the points of the plane, i.e. for any point $P$ of PG$(2,q^4)$ there exists at least one line through $P$ meeting $L_U$ in at least two points.
An easy example is given in the following.

\begin{ex}\label{ex:scattm+1}
Let $L_U$ be a scattered $\fq$-linear set in PG$(2,q^4)$ of rank $5$, that is a linear blocking set. $L_U$ is an example of $1$-saturating linear set.
Indeed, the size of $L_U$ is $q^4+q^3+q^2+q+1$ and the number of lines through any point of the plane is $q^4+1$, so that for any point there exists at least one secant line to $L_U$ passing through it. An explicit example is given by $L_U$, where 
\[ U=\{ (x,x^q,a)\colon x \in \F_{q^4}, a \in \fq \}. \]
\end{ex}

In this section we will prove that the minimal rank of a $1$-saturating linear set in PG$(2,q^4)$ is $5$, answering to Question 4.2 by Bartoli, Borello and Marino in \cite{BBM23}.
Indeed, by \eqref{eq:boundsons} and by Example \ref{ex:scattm+1}, we know that 
\[ 4 \leq s_{q^4/q}(3,2)\leq 5. \]
In \cite[Section 3]{BBM23} the authors proved that if $q$ is  large and even then $s_{q^4/q}(3,2)= 5$. Their proof rely on a careful study of the rational points of algebraic varieties associated with the linear sets of PG$(2,q^4)$ of rank $4$.

In this section we will avoid the algebraic geometry approach and we will prove that for any prime power $q$, $s_{q^4/q}(3,2)=5$. Finally we will also characterize the linear sets of minimum rank giving saturating linear sets in PG$(2,q^4)$.

\begin{rem}
    We will frequently use that the property of a linear set to be $1$-saturating is invariant under projective equivalence. For instance, if $L_U$ is $1$-saturating, $L_{\varphi(U)}$ is $1$-saturating as well, for any $\varphi \in \mathrm{GL}(3,q^4)$.
\end{rem}

We start with the following lemma, in which we prove that a $1$-saturating linear set of rank $4$ cannot be contained in a line.

\begin{lem}\label{lem:step0}
    Suppose that $L_U$ is an $\fq$-linear set of rank $4$ which is $1$-saturating in $\mathrm{PG}(2,q^4)$.
    Then $L_U$ spans the plane.
\end{lem}
\begin{proof}
        If $L_U$ is contained in the line $\ell$, then any point outside $\ell$ is not covered by any secant line to $L_U$. A contradiction.
\end{proof}

We now prove that a $1$-saturating linear set needs to be scattered, with some extra properties.

\begin{lem}\label{lem:step1}
    Suppose that $L_U$ is an $\fq$-linear set of rank $4$ which is $1$-saturating in $\mathrm{PG}(2,q^4)$.
    Then $L_U$ is either scattered with respect to the lines or scattered with the property that all the lines have weight at most two except one of weight three.
\end{lem}
\begin{proof}
    By Lemma \ref{lem:step0}, $L_U$ cannot be contained in any line.
    Therefore, we have that $w_{L_U}(\ell)\leq 3$ for any line $\ell$ of PG$(2,q^4)$. 
    Suppose that there exists a line $r=\mathrm{PG}(W,\F_{q^4})$ such that $w_{L_U}(r)=3$. Then we have that
    \[ U=(W\cap U)\oplus \la u \ra_{\fq}, \]
    where $u \in U \setminus W$. 
    Now, let $\ell=\mathrm{PG}(W',\F_{q^4})$ be a secant line to $L_U$. Since $\ell$ is a secant line, then $w_{L_U}(\ell)\geq 2$, and so \begin{equation}\label{eq:weightaa}\dim_{\fq}(W'\cap (W\cap U))=\dim_{\fq}((W'\cap U)\cap (W\cap U))\geq 1.\end{equation} So, all the secant line to $L_U$ meet $r\cap L_U$.
    Suppose that $L_U$ is not scattered, that is there exists a point $P=\la v \ra_{\F_{q^4}}$ such that $w_{L_U}(P)>1$. It cannot have weight greater than $2$, otherwise $L_U$ would be contained in a line.
    Hence $w_{L_U}(P)=2$ and $P$ lies on a line $r'$ of weight three with respect to $L_U$. Therefore $r=r'$. Note that in $L_U$ cannot exist another point of weight two (otherwise we get again a contradiction from the fact that $L_U$ is not contained in a line), so that $|r\cap L_U|= q^2+1$ by \eqref{eq:pesicard} and \eqref{eq:pesivett}. Taking into account the following facts: 
    \begin{itemize}
        \item the secant lines to $L_U$ are the line $r$ and all the lines joining $L_U\cap r$ with any point of $L_U\setminus r$;
        \item any secant line to $L_U$ has at least $q+1$ points;
        \item any secant line to $L_U$ through $P$ meets $L_U$ is at least $q^2+1$ points;
        \item $|L_U|=q^3+q^2+1$ by Lemma \ref{lem:extensionpoint};
    \end{itemize}
    the number of points of PG$(2,q^4)$ lying on a secant line to $L_U$ is at most
    \[ q^2 \frac{q^3+q^2+1-(q^2+1)}{q}(q^4-q)+\frac{q^3+q^2+1-(q^2+1)}{q^2}(q^4-q^2)+|L_U|\]\[=q^8+q^2+1, \]
    where the first summand is obtained by considering the secant lines to $L_U$ through the points of weight one in $r\cap L_U$ (all but $P$) and we divide by $q$ since any secant line to $L_U$ has at least $q+1$ points; the second one corresponds to the secant to $L_U$ line through the point $P$ and, since we did not consider the points of $L_U$ in the previous two summands, we added it at the end of the sum. 
    Hence, since $q^8+q^2+1$ is strictly less than the number of the points in the plane, we obtain that $L_U$ is not $1$-saturating, a contradiction. To conclude, we need to show that there exists only one line in the plane having weight three. Suppose that there is another line $s$ having weight three in $L_U$. Since the rank of $L_U$ is $4$, arguing as in \eqref{eq:weightaa}, $w_{L_U}(s\cap r)\geq 2$, a contradiction to the scatterdness of $L_U$.
    The assertion is now proved.
\end{proof}

We are now ready to prove the main result, by excluding that linear sets considered in the above lemma are $1$-saturating linear sets.

\begin{thm}\label{main}
    For any prime power $q$, $s_{q^4/q}(3,2)=5$.
\end{thm}
\begin{proof}
    Suppose that $L_U$ is a $1$-saturating $\fq$-linear set of rank $4$. We will first show that, based on Lemma \ref{lem:step1}, we have up to equivalence two examples to exclude in order to prove the non-existence of such a linear set.
    By Lemma \ref{lem:step1}, $L_U$ is either scattered with respect to the lines or scattered having all the lines of weight at most two except one of weight three.
    In the former case, by Corollary \ref{cor:inslinGab}, we obtain that, up to GL$(2,q^4)$-equivalence, 
    \begin{equation}\label{eq:firstform}
    U=\{ (x,x^q,x^{q^2}) \colon x \in \F_{q^4} \}.
    \end{equation}
    Suppose now that $L_U$ is scattered having all the lines of weight at most two except one of weight three, namely $r=\mathrm{PG}(W,\F_{q^4})$.
    We may assume that $V=\F_{q^m}^3$, $r$ is the line of the points having zero as last coordinate and 
    \[ U=(W\cap U)\oplus \la u \ra_{\fq}, \]
    where $u \in U \setminus W$ and $\dim_{\fq}(W\cap U)=3$.
    Since $L_{W\cap U}$ is an $\fq$-linear set of rank three contained in $r$, by Corollary \ref{cor:inslinGab}, $W\cap U$ is GL$(2,q^4)$-equivalent to
    \[ U'=\{ (x,x^q,0) \colon x \in Z \}, \]
    for some $\fq$-subspace $Z$ of $\F_{q^4}$ of dimension three.
    Extending the action of this map to the ambient space $V$ we obtain
    \begin{equation}\label{eq:secondform}
    U=\{ (x,x^q, t u') \colon x \in Z, t \in \fq \},
    \end{equation}
    for some $u' \in \F_{q^4}$.
    Now, either in case of \eqref{eq:firstform} or in case of \eqref{eq:secondform}, we get that through the point $\la (0,0,1)\ra_{\F_{q^4}}$ does not pass any secant line to $L_U$. Indeed, if such a line exists there are $x, y \in \F_{q^4}$ (or in $Z$) not $\fq$-proportional and $\alpha,\beta \in \F_{q^4}^*$ such that
    \[ \alpha (x,x^q)+\beta (y,y^q)=(0,0), \]
    from which we get a contradiction as $x^q$ is a scattered polynomial.
    Therefore, we have proved that $s_{q^4/q}(3,2)>4$. The assertion now follows by Example \ref{ex:scattm+1}.
\end{proof}

In order to characterize the linear sets of rank $5$ that are $1$-saturating in PG$(2,q^4)$ we recall the following. 
The $\fq$-linear sets of rank $5$ in PG$(2,q^4)$ are \emph{blocking sets} in PG$(2,q^4)$, that is they meet any line in at least one point. Starting from the paper of Polito and Polverino \cite{PoPo}, such a linear sets have been classified in \cite{BoPol} by Bonoli and Polverino, and an explicit description of those of R\'edei type (those having a line of weight four) has been given in \cite{BMZZ,Complwei}.

To our aim we just need the following, which is a consequence of the results in \cite{BoPol}.

\begin{lem}\label{lem:sizerank5}
    Let $L_U$ be an $\fq$-linear set of rank $5$ in PG$(2,q^4)$. Then one of the following occurs
    \begin{itemize}
        \item $|L_U|=q^4+1$ and $L_U$ is a line;
        \item $|L_U|=q^4+q^2+1$ and $L_U$ is a Baer subplane;
        \item $|L_U|=q^4+q^3+1$;
        \item $|L_U|=q^4+q^3+q^2+1$;
        \item $|L_U|=q^4+q^3+q^2-q+1$;
        \item $|L_U|=q^4+q^3+q^2+q+1$.
    \end{itemize}
    In particular, if $L_U$ is not a line then $|L_U|>q^4+1$.
\end{lem}

As a consequence we get that all the linear sets of rank $5$, different from a line, are $1$-saturating.

\begin{cor}
    Let $L_U$ be an $\fq$-linear set of rank $5$ in $\mathrm{PG}(2,q^4)$. Then it is $1$-saturating linear set if and only if $L_U$ is not a line.
\end{cor}
\begin{proof}
    If $L_U$ is $1$-saturating linear set then it cannot be contained in a line by Lemma \ref{lem:step0}. Suppose now that $L_U$ is not a line, then Lemma \ref{lem:sizerank5} implies that $|L_U|>q^4+1$. Arguing now as in the Example \ref{ex:scattm+1}, the number of lines trough any point of the plane is $q^4+1$, so that for any point there exists at least one secant line to $L_U$ passing through it. Therefore $L_U$ is a $1$-saturating linear set.
\end{proof}

A complete list, up to equivalence, of these linear sets is given in \cite[Section 5]{BoPol}.

\section*{Acknowledgements}

The author is very grateful to Daniele Bartoli, Martino Borello and Giuseppe Marino for fruitful discussions.
The research was supported by the project ``COMBINE'' of the University of Campania ``Luigi Vanvitelli'' and was partially supported by the Italian National Group for Algebraic and Geometric Structures and their Applications (GNSAGA - INdAM).
This research was also supported by Bando Galileo 2024 – G24-216 and by the project ``The combinatorics of minimal codes and security aspects'', Bando Cassini.

Ferdinando Zullo\\
Dipartimento di Matematica e Fisica,\\
Universit\`a degli Studi della Campania ``Luigi Vanvitelli'',\\
Viale Lincoln 5,\\
81100 Caserta CE - Italy\\
{{\em ferdinando.zullo@unicampania.it}}

\end{document}